\documentclass[11pt]{article}
\usepackage{}

\usepackage[total={6in, 8.5in}]{geometry}
\usepackage{amsfonts}
\usepackage{amsthm}
\usepackage{amssymb}
\usepackage{amsmath}
\usepackage{enumerate}
\usepackage[
pdfauthor={ESYZ},
pdftitle={overfull with small core},
pdfstartview=XYZ,
bookmarks=true,
colorlinks=true,
linkcolor=blue,
urlcolor=blue,
citecolor=blue,
bookmarks=false,
linktocpage=true,
hyperindex=true
]{hyperref}

\makeatletter
\newcommand{\setword}[2]{%
	\phantomsection
	#1\def\@currentlabel{\unexpanded{#1}}\label{#2}%
}
\makeatother

\makeatletter
\renewcommand*\env@matrix[1][*\c@MaxMatrixCols c]{%
	\hskip -\arraycolsep
	\let\@ifnextchar\new@ifnextchar
	\array{#1}}
\makeatother

\usepackage{graphicx}
\usepackage[natural]{xcolor}
\usepackage{tikz}
\usepackage{tikz-3dplot}
\usetikzlibrary{shapes}
\usetikzlibrary{arrows,decorations.pathmorphing,backgrounds,positioning,fit,petri,automata}
\usetikzlibrary {positioning}
\usetikzlibrary{arrows}

\usepackage{pgf,tikz,pgfplots}

\usepackage{mathrsfs}

\usetikzlibrary{arrows}

\usepackage{xcolor}
\long\def\ignore#1{}

\let\oldi\ignore

\newcommand{\D}{\Delta}

 \oldi{
\newtheorem{THM}{\textbf{Theorem}}[section]
\newtheorem{THMs}{\textbf{Theorem}}[section]
\newtheorem{DEF}[THM]{\textbf{Definition}}[section]
\newtheorem{LEM}[THM]{\textbf{Lemma}}
\newtheorem{CON}[THM]{\textbf{Conjecture}}
\newtheorem{PROP}[THM]{\textbf{Proposition}}
\newtheorem{COR}[THM]{\textbf{Corollary}}
\newtheorem{CORs}{\textbf{Corollary}}[section]
\newtheorem{PRO}[THM]{\textbf{Problem}}

\newtheorem{FAC}{\textbf{Fact}}
\newtheorem{REM}{\textbf{Remark}}
\newtheorem{OPR}{\textbf{Operation}}
\newtheorem{CLA}{\textbf{Claim}}[section]
\setcounter{CLA}{1}
 }

\newtheorem{THM}{Theorem}[section]

\newtheorem{LEM}[THM]{Lemma}

\usepackage{thmtools}
\usepackage{thm-restate}

\newcommand{\CC}{\mathcal{C}}

\newcommand{\phiv}{\varphi}
\newcommand{\phibar}{\overline{\varphi}}

\newcommand{\pbar}{\overline{\varphi}}

\begin{document}
\title{Overfullness of edge-critical graphs with small minimal core degree}

\author{%
	Yan Cao\\
	School of Mathematical and Data Sciences, \\
	West Virginia University, Morgantown, WV 26506\\
	\texttt{yacao@mail.wvu.edu}
	\and
	Guantao Chen\thanks{This work was supported in part by NSF grant DMS-1855716.}\\
	Department of Mathematics and Statistics, \\
	Georgia State University, Atlanta, GA 30302\\
	\texttt{gchen@gsu.edu}%
	\and
	Guangming Jing\thanks{This work was supported in part by NSF grant DMS-2001130.}\\
	School of Mathematical and Data Sciences,\\
West Virginia University, Morgantown, WV 26506\\
	\texttt{gujing@mail.wvu.edu}%
	\and 
	Songling Shan\thanks{This work was supported in part by NSF grant DMS-2153938.}\\
	Department of Mathematics, \\
	Illinois State Univeristy, Normal, IL 61790 \\
	\texttt{sshan12@ilstu.edu}
} 

\date{\today}
\maketitle

 \begin{abstract}
 Let $G$ be a simple graph. Denote by $n$, $\Delta(G)$ and $\chi' (G)$ be the order, the maximum degree and the chromatic index of $G$, respectively.   We call $G$ 
  \emph{overfull} if $|E(G)|/\lfloor n/2\rfloor > \D(G)$, and {\it critical} if $\chi'(H) < \chi'(G)$ for every proper subgraph $H$ of $G$.  Clearly, if $G$ is overfull then $\chi'(G) = \Delta(G)+1$. The \emph{core} of $G$, denoted by $G_{\D}$, is the subgraph of $G$ induced by all its maximum degree vertices.  We believe that utilizing the core degree condition could be considered as an approach to attacking the overfull conjecture.  Along this direction, we in this paper show that for any integer $k\geq 2$, if $G$ is critical with $\Delta(G)\geq \frac{2}{3}n+\frac{3k}{2}$ and $\delta(G_\Delta)\leq k$, then $G$ is overfull. 

 \end{abstract}


\section{Introduction}
We will mainly follow the notation from~\cite{StiebSTF-Book}. Graphs in this paper are simple, i.e., finite, undirected, without loops or  multiple edges.  Let $G$ be a graph and let $[k]=\{i\ |\ 1\leq i\leq k\ and\ i\in\mathbb{Z}\}$ for a nonnegative integer $k$.  A {\it $k$-edge-coloring\/} of $G$ is a mapping $\varphi$:  $E(G) \rightarrow [k]$   that assigns to every edge $e$ of $G$ a color $\varphi(e) \in [k]$  such that  no two adjacent edges receive the same color.  Denote by $\CC^k(G)$ the set of all $k$-edge-colorings of $G$. 
The {\it chromatic index\/} $\chi'(G)$ is  the least integer $k\ge 0$ such that $\CC^k(G) \ne \emptyset$.  Denote by $\delta(G)$ and
$\D(G)$ the minimum and maximum degree of $G$, respectively. 
In 1960's, Vizing~\cite{Vizing-2-classes} and, independently,  Gupta~\cite{Gupta-67} proved that $\D(G) \le \chi'(G) \le \D(G)+1$. 
This leads to a natural classification of graphs. Following Fiorini and Wilson~\cite{fw},  we say a graph $G$ is of {\it class 1} if $\chi'(G) = \D(G)$ and of \emph{class 2} if $\chi'(G) = \D(G)+1$.
Holyer~\cite{Holyer} showed that it is NP-complete to determine whether an arbitrary graph is of class 1.

A graph G is \emph{critical} if $\chi'(H)<\chi'(G)$ for every proper subgraph $H$ of $G$.
In investigating the Classification Problem, critical graphs are of particular interest. 
A critical class 2 graph is called  \emph{$\D$-critical} if
 $\D(G) = \D$. An edge $e\in E(G)$ is called a {\it critical} edge if $\chi'(G-e)<\chi'(G)$. Clearly, if $G$ is critical, then every edge of $G$ is a critical edge. For convenience,  we denote $|V(G)|$ by $n$ throughout this paper. Since every matching of  $G$ has at most 
 $\lfloor n/2\rfloor$ edges,  $\chi'(G) \ge |E(G)|/ \lfloor n/2\rfloor$. 
 A graph $G$ is {\em overfull} if $|E(G)| / \lfloor n/2\rfloor > \D(G)$. In 1986, Chetwynd and Hilton~\cite{MR848854} conjectured that if $G$ is a class 2 critical graph with $\Delta(G)> \frac{n}{3}$, then $G$ is overfull. This conjecture is known as the Overfull Conjecture.
  
 The \emph{core} of a graph $G$, denoted by $G_{\D}$, is the subgraph induced by all its maximum degree vertices.
 Vizing~\cite{Vizing-2-classes} proved that if $G_\Delta$ has at most two vertices then $G$ 
is class 1. Fournier~\cite{MR0349458} generalized Vizing's result by showing that 
if $G_\Delta$ is acyclic then $G$ is class 1. 
Thus a necessary condition for a graph to be class 2 is to have a core
that contains cycles. A long-standing conjecture of Hilton and Zhao~\cite{MR1395947} claims that for a connected graph $G$ with $\D\ge 4$, 
if $\Delta(G_{\D})\leq 2$, then $G$ is overfull. We~\cite{HZ} recently confirmed this conjecture.  Another paper of Cao, Chen, and Shan~\cite{core1} extended the result above by changing the maximum degree core condition to a minimum core degree condition,  and showed that for any critical class 2 graph $G$, if $\delta(G_{\D})\leq 2$ and $\D(G) > n/2 +1$, then $G$ is overfull.  

Along this direction, we prove the following result and verify the overfull conjecture for critical graphs with a more general minimum core degree condition, and we hope to use similar ideas to attack the overfull conjecture in the future. For example, if we can improve the coefficient of $k$ in Theorem~\ref{main} from $3/2$ to $1/12$, then the overfull conjecture holds for all graphs with maximum degree $\Delta\geq 3n/4$.

\begin{THM}\label{main}
	Let $k\geq 2$ be a positive integer and $G$ be a $\Delta$-critical graph of order $n$. If $\Delta\geq \frac{2}{3}n+\frac{3k}{2}$ and $\delta(G_{\Delta})\leq k$, then $G$ is overfull. 
\end{THM}


%


\section{Preliminaries}\label{lemma}
This section is divided into two subsections. In the first subsection we introduce some basic notation and terminologies. In the second subsection we introduce the traditional Vizing Fan and generalize it to a larger structure.

\subsection{Basic notation and terminologies}

Let $G$ be a graph with maximum degree $\D$, and  let $e\in E(G)$ and $\varphi \in \CC^{\D}(G-e)$.

For a vertex  $v\in V(G)$, define the two color sets
\[
\varphi(v)=\{\varphi(f)\,:\, \text{$f\ne e$ is incident to $v$}\} \quad \mbox{ and}\quad \pbar(v)=[1, \D] \setminus\varphi(v).
\]
We call $\varphi(v)$ the set of colors \emph{present} at $v$ and $\pbar(v)$
the set of colors \emph{missing} at $v$. 
If $|\pbar(v)|=1$, we will also use $\pbar(v)$ to denote the  color  missing at $v$. Let $N(v)$ be the collection of all the neighbors of $v$, $N_{<\Delta}(v)$ be the collection of neighbors of $v$ with degree less than $\Delta$, and $N_{\Delta}(v)$ be the collection of neighbors of $v$ with degree exactly $\Delta$.

For a vertex set $X\subseteq V(G)$,  define  $\pbar(X)=\bigcup _{v\in X} \pbar(v)$ to be the set of missing colors of $X$. 
The set $X$ is called \emph{elementary} w.r.t. $\varphi$  or simply \emph{elementary} if $\pbar(u)\cap \pbar(v)=\emptyset$
for every two distinct vertices $u,v\in X$. 
In the rest of this paper, we may not always mention the coloring $\varphi$ if it is clearly understood.  

For a color $\alpha$, the edge set $E_{\alpha} = \{ e\in E(G)\, |\, \varphi(e) = \alpha\}$ is called a {\it color class}. Clearly, 
$E_{\alpha}$ is a {\it matching} of $G$ (possibly empty). 
For two distinct colors $\alpha,\beta$,  the  subgraph of $G$
induced by $E_{\alpha}\cup E_{\beta}$ is a union of disjoint 
paths and  even cycles, which are referred to as   \emph{$(\alpha,\beta)$-chains} of $G$
w.r.t. $\varphi$.  For a vertex $v$, let $C_v(\alpha, \beta, \varphi)$ denote the unique $(\alpha, \beta)$-chain 
containing $v$.  
If $C_v(\alpha, \beta, \varphi)$ is a path, we just write it as $P_v(\alpha, \beta, \varphi)$. The latter is commonly used when we know that $|\pbar(v)\cap \{\alpha,\beta\}|=1$.  If we interchange the colors $\alpha$ and $\beta$
on an $(\alpha,\beta)$-chain $C$ of $G$, we briefly say that the new coloring is obtained from $\varphi$ by an 
{\it $(\alpha,\beta)$-swap} on $C$, and we write it as  $\varphi/C$. 
This operation is called a \emph{Kempe change}.  If $\alpha\in\phibar(v)$, by doing operation $\alpha\rightarrow \beta$ at $v$ we mean the Kemp change $\varphi/P_{v}(\alpha,\beta,\varphi)$. We say two vertices $x$ and $y$ are {\em $(\alpha,\beta)$-linked} if they belong to the same $(\alpha,\beta)$-chain. Moreover, when $x=y$, for convenience we still say $x$ and $y$ are $(\alpha,\beta)$-linked even if $\alpha,\beta\in\phibar(x)$.

\subsection{Linear Sequence, Shifting, and Extended Multi-fan}

The fan argument was introduced by Vizing~\cite{Vizing64,vizing-2factor} in his proof of the classic results on the upper bounds of chromatic indices. 
  Let  $G$ be a class 2 graph with maximum degree $\D$, $e=rs$ be a critical edge of $G$, and let $\varphi\in \CC^{\D}(G-e)$. For an integer $p\geq0$, a sequence
$F=(r,e_0,s_0,e_1, s_1, \ldots, e_p, s_p)$ alternating between distinct vertices and edges is called a
{\em multi-fan} at $r$ with respect to $e$ and $\varphi$ if $s_0=s$, $e_0=e$ and for each $i\in [p]$, the edge $e_i=(r,s_i)$  satisfies 
$\phiv(e_i) \in \overline{\varphi}(s_j)$ for some $0\leq j\leq i-1$. For the purpose of generalization in this paper, we include the vertex $r$ in $F$ comparing to the definition of a multi-fan in the book~\cite{StiebSTF-Book}.  Let $q$ be a nonnegative integer. A  {\em linear sequence} at $r$ from $s_0$ to $s_q$ in $G$, denoted by $L=(r,e_0,s_0, e_1,s_1,\dots,e_q,s_q)$, is a sequence of distinct vertices and edges such that $e_i\in E(G)$ for $0\leq i\leq q$ and $\varphi(e_i)\in\overline\varphi(s_{i-1})$ for $i\in [q]$. Denote by  $V(L)$ and $E(L)$ respectively the set of vertices and edges contained in $L$. A {\em shifting} from $s_i$ to $s_j$ in the linear sequence $L=(r,e_0,s_0,e_1,s_1,\dots,e_q,s_q)$ is an operation that replaces the current color of $e_t$ by the color of $e_{t+1}$ for each $i\le t\leq j-1$ with $0\leq i\leq j\leq q$. Note that shifting does not change the color of $e_j$ where $e_j=xs_j$, so the resulting coloring will not be a proper coloring. In our proof we will treat $e_j$ separately to avoid this problem.

 Note that $e_0$ may not be $e$ in a linear sequence, but a linear sequence with $e_0=e$ is also a  multi-fan at $r$. Moreover, for any $s_i\in V(F)$ with $i\in [p]$, the multi-fan $F=(r,e_0,s_0,e_1, s_1, \ldots, e_p, s_p)$ contains a linear sequence at $r$ from $s_0$ to $s_i$. A linear sequence at $r$ with $\varphi(e_0)=\tau$ is called a {\em ${\tau}$-sequence}.  In our proof we will add some linear sequences not contained in a multi-fan to enlarge it.  We say a multi-fan $F$ at $r$ is {\em maximal} w.r.t. $e$ and $\varphi$ if there is no multi-fan at $r$ w.r.t. $e$ and $\varphi$ containing $F$ as a proper sub-sequence. We say a multi-fan $F$ at $r$ is  {\em maximum} w.r.t. $e$ if $|V(F)|$ is maximum among all multi-fans at $r$ w.r.t. $e$ over all colorings $\varphi\in \CC^{\D}(G-e)$. Clearly if $F$ is maximum at $r$ w.r.t. $e$, it is also maximal w.r.t. $e$ and the coloring $\varphi$ where $F$ is obtained. Let $F$ be a maximal multi-fan at $r$ w.r.t. $e$ and $\varphi$. A $\tau$-sequence $L$ at $r$ is said to be outside of $F$ if $V(L)\cap V(F)=\{r\}$. For an integer $t\geq 0$, we say a $\tau$-sequence $L=(r,f_0,v_0,f_1,v_1,\dots,f_t,v_t)$ at $r$ outside of $F$ is {\em extremal} 
 if $v_t$ is the only vertex $v_j$ with index $0\leq j\leq t$ such that either $\phibar(v_j)\cap(\cup_{i=0}^{j-1}\phibar(v_i)\cup\phibar(V(F))\cup\{\tau\})\neq \emptyset$ or $\phibar(v_j)=\emptyset$. Since a $\tau$-sequence cannot be enlarged forever, it must be a subsequence of some extremal $\tau$-sequence. Moreover, exactly one of the followings must happen for an extremal $\tau$-sequence $L$:
 \begin{enumerate}[(a)]
 	\item $V(L)\cup V(F)$ is elementary and $\{\tau\}=\phibar(v_t)$. In this case we say $L$ is of {\em Type A}.
 	\item $\phibar(v_i)\cap \phibar(V(F))\neq\emptyset$ for some $0\leq i\leq t$. In this case we say $L$ is of {\em Type B}.
 	\item $\phibar(v_i)\cap \phibar(V(F))=\emptyset$ for all $0\leq i\leq t$, and $V(L)$ is not elementary. In this case there exists a color $\alpha\in(\phibar(v_i)\cap\phibar(v_j))-\phibar(V(F))$ for some $0\leq i\leq j\leq t$ and we say $L$ is of {\em Type C}.
 	\item $\phibar(v_t)=\emptyset$ and $V(L)\cup V(F)$ is elementary. In this case $d(v_t)=\Delta$ and  we say $L$ is of {\em Type D}. 
 \end{enumerate}

 From now on we will not mention ``at $r$'' when we refer to a multi-fan or a linear sequence if it creates no confusion. Additionally, when we refer to a linear sequence outside of $F$, we always mean an extremal one unless specified otherwise. The following result regarding a multifan can be found in \cite[Theorem~2.1]{StiebSTF-Book}.

\begin{LEM}
	\label{vf1}
	Let $G$ be a class 2 graph, $e=rs_0$ be a critical edge and $\varphi\in \CC^\Delta(G-e)$. 
	If $F$  is a multifan w.r.t. $e$ and $\varphi$,   then  $V(F)$ is elementary. 
\end{LEM}
Let $G$ be a class 2 graph, $e=rs_0$ be a critical edge and $\varphi\in \CC^\Delta(G-e)$. Let $F=(r,e_0,s_0,e_1, s_1, \ldots, e_p, s_p)$ be a multi-fan centered at $r$ under the coloring $\varphi$. Clearly a linear sequence $L$ at $r$ from $s_0$ to $s_q$ with $q\in [p]$ defines a linear order $\preceq_L$. Since $V(F)$ is elementary by Lemma~\ref{vf1}, it's easy to see that all the linear sequences at $r$ starting from $s_0$ to $s_q$ for some $q\in [p]$ together induce a partial order $\preceq_F$ by $\alpha\preceq_F \beta$ for two colors $\alpha,\beta\in \phibar(V(F))$ if and only if  there exists a linear sequence $L$ at $r$ starting from $s_0$ to some $s_q$ with $q\in [p]$ such that an edge $e'\in E(L)$ with $\varphi(e')=\alpha$ comes before a vertex $v\in V(L)$ with $\beta\in\phibar(v)$ along $L$. Moreover, for any color $\alpha\in\phibar(V(F))$, there is a unique vertex $v\in V(F)$ such that $\alpha\in\phibar(v)$ since $V(F)$ is elementary.  Let $v_F({\alpha})$ denote such vertex $v$. 
 We have the following Lemma as a direct consequence of Lemma~\ref{vf1}.
\begin{LEM}
	\label{vf2}
Let $G$ be a class 2 graph, $e=rs_0$ be a critical edge and $\varphi\in \CC^\Delta(G-e)$. Let  $F$ be a maximal multi-fan at $r$ w.r.t. $e$ and $\varphi$. Then for any two colors $\alpha,\beta\in\phibar(V(F))$, we have the following statements:
	
		\begin{enumerate}[(a)]
	\item If $v_F(\alpha)=r$, then $v_F(\alpha)$ and $v_F(\beta)$ are $(\alpha,\beta)$-linked.

	\item If $\alpha$ and $\beta$ are incomparable along $\preceq_F$, then $v_F(\alpha)$ and $v_F(\beta)$ are $(\alpha,\beta)$-linked.
	\item If $\alpha\preceq_F\beta$ along $\preceq_F$ and  $v_F(\alpha)$ and $v_F(\beta)$ are not $(\alpha,\beta)$-linked, then $P_{v_F(\beta)}(\alpha,\beta,\varphi)$ must contain the vertex $r$.
	\item If $v\in V(F)$ and $v\neq r$, then $F$ contains at least $|\phibar(v)|$ many $\Delta$-degree neighbors of $r$.

	\end{enumerate}
\end{LEM}
\begin{proof}
To prove (a) we assume $v_F(\alpha)=r$. Note that if $v_F(\beta)=r$, we are done by definition. So we may assume that $v_F(\beta)\neq r$. If $v_F(\alpha)$ and $v_F(\beta)$ are not $(\alpha,\beta)$-linked, then by $\beta\rightarrow \alpha$ at $v_F(\beta)$, we have a non-elementary multi-fan contradicting Lemma~\ref{vf1}. Thus (a) holds.

For (b) we assume that $\alpha$ and $\beta$ are incomparable along $\preceq_F$. Note that there are two linear sequences $L_1$ and $L_2$ at $r$ from $s_0$ to $v_F(\alpha)$ and $v_F(\beta)$, respectively. Since $\alpha$ and $\beta$ are incomparable along $\preceq_F$, and $v_F(\alpha)$ and $v_F(\beta)$ are the last vertices for $L_1$ and $L_2$ respectively, $L_1$ and $L_1$ does not contain any edge colored by $\alpha$ or $\beta$. Now by $\beta\rightarrow \alpha$ at $v_F(\beta)$, we have a new coloring and we denote the new coloring by $\varphi_1$. Since no edge in $L_1$ and $L_2$ is colored by either $\alpha$ or $\beta$ under $\varphi$,  $L_1$ and $L_2$ are still linear sequences under $\varphi_1$. Let $F'$ be a maximal multi-fan w.r.t. $e$ and $\varphi_1$. Then $L_1$ and $L_2$ are all contained in $F'$, giving a non-elementary multi-fan contradicting  Lemma~\ref{vf1}.

If (c) fails, since $\alpha\preceq_F\beta$, we can just do $\beta\rightarrow\alpha$ at $v_F(\beta)$ to get a non-elementary multi-fan contradicting Lemma~\ref{vf1}.

To see (d), we assume $\alpha\in\phibar(v)$ with $v\in V(F)$ and $v\neq r$. Since $V(F)$ is elementary by Lemma~\ref{vf1} and $F$ is maximal, every color $\alpha$ in $\phibar(v)$ induces at least one $\alpha$-sequence containing a unique $\Delta$-degree vertex in $F$, giving at least $|\phibar(v)|$ many $\Delta$-degree neighbors of $r$ in $F$. 
\end{proof}	
The following Vizing's Adjacency Lemma is a direct consequence of Lemmma~\ref{vf2}(d).

\begin{LEM}[Vizing's Adjacency Lemma(VAL)] Let $G$ be a class 2 graph with maximum degree $\Delta$. If $e=xy$ is a critical edge of $G$, then $x$ is adjacent to at least $\Delta-d(y)+1$ $\Delta$-vertices from $V(G)\setminus \{y\}$.
	\label{thm:val}
\end{LEM}

Let $G$ be a class 2 graph, $e=rs_0$ be a critical edge and $F=(r,e_0,s_0,e_1, s_1, \ldots, e_p, s_p)$ be a maximum multi-fan at $r$ w.r.t. $e$, and let $\varphi\in \CC^\Delta(G-e)$ be the coloring where $F$ is obtained. We call a color $\beta$ a {\em stopping} color at $r$ if $r$ has a $\Delta$-degree neighbor $x$ with $\varphi(rx)=\beta$. Let $K$ be the set of all stopping colors at $r$.  Since $F$ is maximum and elementary, there exist some vertex $s_h\in V(F)$ and a stopping color $\beta$ such that $\beta\in\phibar(s_h)$. We let $K_F=K-\phibar(V(F))$ and call colors in $K_F$ {\em stopping colors outside} of $F$. By a slightly abuse of notation, in this paper, a {\it union} of two sequences $A$ and $B$, denoted by $A\cup B$, is the sequence obtained by joining the sequence $B$ to $A$ after the last element of $A$. We now fix a vertex $s_h\in V(F)$ with a stopping color $\beta\in\phibar(s_h)$. Let $F'$ be the union of all the $\varphi(rv)$-sequences outside of $F$, where $v$ is any vertex in the set $N_{<\Delta}(s_h)\cap N(r)$ with $\varphi(vs_h)\notin K_F$.  Then we call the sequence $F\cup F'$ an {\em extended multi-fan} w.r.t. $F$ and $s_h$. For simplification of notation, we did not indicate $s_h$ in the notation $F'$, even though $F'$ relies on a fixed vertex $s_h$. The following Lemma~\ref{extend} is a key lemma in our proof and it is a natural generalization of Lemmas~\ref{vf1} and~\ref{vf2} on $F\cup F'$. It is worth pointing out that Lemma~\ref{extend} can be easily generalized further along this direction if we allow $F'$ to be the union of all the $\varphi(rv)$-sequences outside of $F$ such that $v\in N_{<\Delta}(s_h)\cap N(r)$ with $\varphi(vs_h)\notin K_F$ for every vertex $s_h$ having any stopping color $\beta\in\phibar(s_h)$.
 
\begin{LEM}
	\label{extend}
Let $G$ be a class 2 graph, $e=rs_0$ be a critical edge and $F=(r,e_0,s_0,e_1,$ $ s_1, \ldots, e_p, s_p)$ be a maximum multi-fan at $r$ w.r.t. $e$, and let $\varphi\in \CC^\Delta(G-e)$ be the coloring where $F$ is obtained.  Let  $F\cup F'$ be an extended multi-fan w.r.t. $F$ and $s_h\in V(F)$, where $\beta$ is a stopping color with $\beta\in \phibar(s_h)$.  Then the following holds.
\begin{enumerate}[(a)]
	\item $V(F\cup F')$ is elementary.
	\item For two colors $1\in\phibar (r)$ and $\gamma\in\phibar(F\cup F')-K_F$, the vertices $r$ and $v_{F\cup F'}(\gamma)$ are $(1,\gamma)$-linked, where  $v_{F\cup F'}(\gamma)$ is the unique vertex in $V(F\cup F')$ with $\gamma\in\phibar(v_{F\cup F'}(\gamma))$.
	\item For each color $\gamma\in\phibar(V(F'))$, the vertices $s_h$ and $v_{F'}(\gamma)$ are $(\beta,\gamma)$-linked, where $v_{F'}(\gamma)$ is the unique vertex in $V(F')$  with $\gamma\in\phibar(v_{F'}(\gamma))$.
	\item Let $\gamma$ be a color in $\phibar(V(F'))\cap K_F$ and let $L=(r,rv_0,v_0,...,v_t)$ be a $\varphi(rv_0)$-sequence in $F'$ such that  $v_{F'}(\gamma)\in V(L)$. If $\varphi(s_hv_0)\neq 1$, then $r$ is $(1,\gamma)$-linked with $v_{F'}(\gamma)$, where $t\geq 0$ is an integer. If $\varphi(s_hv_0)=1$, then  $v_{F'}(\gamma)$ is $(\zeta,\gamma)$-linked with $v_F(\zeta)$ for any color $\zeta\in\phibar(V(F))\cap K$.
\end{enumerate}
\end{LEM}

The proof of Lemma~\ref{extend} will be given in Section~\ref{pextend}. 
We call a vertex $r$ \emph{light} if $d(r)=\Delta$ and $d_{G_\D}(r)=\delta(G_\D)$. The next lemma is the main tool used in the proof of Theorem~\ref{main}.
\begin{LEM}\label{1}
	Let $G$ be a critical class 2 graph with $\delta(G_\D)=k$, $|V(G)|=n$, and $\Delta\geq \frac{2}{3}n+\frac{3k}{2}$, let $r$ be a light vertex, and let $e=rs$ be a critical edge with $d(s)\leq\Delta-1$. Let $\varphi\in \CC^{\D}(G-e)$ be a coloring under which $F$ is a maximum multi-fan centered at $r$. Then all vertices of degree at least $\Delta-k+1$ form an elementary set under $\varphi$. 
\end{LEM}

\section{Proof of Theorem~\ref{main}}
 \begin{proof}
 		Let $G$ be a $\Delta$-critical graph of order $n$ and $k\geq 2$ be a positive integer. Furthermore, we assume $\Delta\geq \frac{2}{3}n+\frac{3k}{2}$ with $\delta(G_{\Delta})\leq k$. Since $\Delta\geq\frac{2}{3}n+\frac{3k}{2}\geq\frac{2}{3}n+\frac{3\delta(G_{\Delta})}{2}$, we will just take $\delta(G_{\Delta})=k$ in this proof. Let $r$ be a light vertex of $G$ and $s$ be a neighbor of $r$ with $d(s)\leq \Delta-1$. Then the edge $rs$ is a critical edge. Let $\varphi$ be a $\Delta$-edge-coloring of $G-rs$ and $F$ be a maximum multi-fan centered at $r$. We first claim that if $V(G)$ is elementary under $\varphi$, then $G$ is overfull. Indeed, if $G$ is elementary, then each color can only be missing at most once for vertices in $V(G)$. Since $r$ has at least one missing color, $n$ must be odd as any color missing at $r$ induces a matching of $G-r$. Therefore, each color must be missing exactly once in $G$ as $n$ is odd. Thus $G$ has exactly $(\frac{n-1}{2})\Delta+1$ many edges since we have $\Delta$ many color classes and the edge $rs$ is uncolored. So $G$ is overfull as we claimed. 
 		
 		Now we shall show that $V(G)$ is elementary to confirm that $G$ is overfull in the reminder of this section.  By Lemma~\ref{1}, all vertices with degree at least $\Delta-k+1$ form an elementary set, so we are done if there's no vertex of degree less than $\Delta-k+1$. Thus we assume otherwise that there is a vertex $x$ with $d(x)\leq \Delta-k$. Since $|N_{\Delta}(r)|=k$, all the vertices in $N(r)$ have degree at least $\Delta-k+1$ by Lemma~\ref{thm:val}(VAL). So $x\notin N(r)$. 
 		
 		We claim that $d(x)\geq \frac{n}{3}+2k$. Since every edge in $G$ is critical, $x$ is adjacent to at least one maximum degree vertex in $G$ by VAL. Let $u$ be a maximum degree vertex with $ux\in E(G)$. Then $u\neq r$ as $x\notin N(r)$. Since $d(u)=\Delta$, we have $|N(u)\cap N(r)|\geq d(u)+d(r)-|N(u)\cup N(r)|\geq \Delta+\Delta-n\geq \frac{4n}{3}+3k-n\geq \frac{n}{3}+3k$. Since $|N_{\Delta}(r)|=k$, we have $|N_{<\Delta}(u)|\geq \frac{n}{3}+2k$, and therefore $|N_{\Delta}(u)|\leq \Delta-\frac{n}{3}-2k$. Since $ux$ is a critical edge, we have $|N_{\Delta}(u)|\geq\Delta-d(x)+1$. So $d(x)\geq \frac{n}{3}+2k$ as claimed.
 		
 		Since $N_{\Delta}(r)=k$, we have $d(v)\geq\Delta-k+1$ for each vertex $v\in N_{<\Delta}(r)$ by VAL. Recall that by Lemma~\ref{1}, all vertices with degree at least $\Delta-k+1$ form an elementary set.  As $s\in N_{<\Delta}(r)$, we have $|\phibar(N_{<\Delta}(r))|\geq |N_{<\Delta}(r)|+1\geq\Delta-k+1$. Since $d(x)\geq \frac{n}{3}+2k$, we have $|N(r)\cap N(x)|\geq \Delta+\frac{n}{3}+2k-n\geq \frac{7k}{2}$. Because $|N_{\Delta}(r)|=k$, it follows that $|N_{<\Delta}(r)\cap N_{<\Delta}(x)|\geq \frac{5k}{2}$. Since $|\phibar(N_{<\Delta}(r))|\geq\Delta-k+1$ and all edges between $x$ and vertices $N_{<\Delta}(r)\cap N_{<\Delta}(x)$ are colored differently, 
 		there is a vertex $v\in N_{<\Delta}(r)\cap N_{<\Delta}(x)$ such that $\varphi(vx)=\beta\in\phibar(w)$ where $w\in N_{<\Delta}(r)$. Since $d(x)\leq \Delta-k$, $|\phibar(x)|\geq k$. Since $|\phibar(N_{<\Delta}(r))|\geq \Delta-k+1$, $|\phibar(x)\cap \phibar(N_{<\Delta}(r))|\geq 1$. Thus, there exists $\alpha\in\phibar(x)\cap\phibar(u')$ where $u'\in N_{<\Delta}(r)$. So $d(u')\geq \Delta-k+1$. Let $1\in\phibar(r)$.  We claim that $u'$ is $(1,\alpha)$-linked with $r$. Otherwise, we have $u'\notin V(F)$ by Lemma~\ref{vf2}(a). Thus $F$ stays as a maximum multi-fan after $1\rightarrow\alpha$ at $u'$. However, we have a contradiction to Lemma~\ref{1} as now $1\in\phibar(r)\cap\phibar(u')$ and $d(u')\geq \Delta-k+1$. Thus we have as claimed. Then $x$ is not $(1,\alpha)$-linked with $r$, and $x\notin V(F)$ by Lemma~\ref{vf2}(a). Thus after doing $\alpha\rightarrow 1$ at $x$, $F$ stays as a maximum multi-fan.  Now we let $\gamma\in\phibar(v)$. Similarly as earlier, we see that $v$ is $(1,\gamma)$-linked with $r$, as otherwise we can do $\gamma\rightarrow 1$ at $v$ and reach a contradiction with Lemma~\ref{1}. Hence $x\notin V(F)$ by Lemma~\ref{vf2}(a). Thus we do $1\rightarrow \gamma$ at $x$. Now $\gamma\in\phibar(x)\cap\phibar(v)$ and we recolor the edge $vx$ by $\gamma$. Note that $F$ stays as a maximum multi-fan after these two operations. As a result, we have $\beta\in\phibar(v)\cap\phibar(w)$, a contradiction to Lemma~\ref{1}. Therefore, $G$ has no vertex of degree less than $\Delta-k+1$ and $G$ is elementary by Lemma~\ref{1}.
 	
 \end{proof}	

\section{Proof of Lemma~\ref{extend}}\label{pextend}
\begin{proof}

Let $G$ be a class 2 graph, $e=rs_0$ be a critical edge and $F=(r,e_0,s_0,e_1, s_1, \ldots, e_p, s_p)$ be a maximum multi-fan at $r$ w.r.t. $e$, and let $\varphi\in \CC^\Delta(G-e)$ be the coloring where $F$ is obtained. Let $F\cup F'$ be an extended multi-fan as defined earlier using the vertex $s_h$ with $0\leq h\leq p$. Let $1\in\phibar(r)$ and $\beta\in\phibar(s_h)\cap K$.

We first prove (a).  Assume otherwise that there exist $\alpha\in \phibar(x_1)\cap\phibar(x_2)$ with $x_1,x_2\in V(F\cup F')$. 
We first assume that $\alpha\notin K_F$. Since only one of $x_1,x_2$ is $(1,\alpha)$-linked with $r$, so we assume that $x_1$ is not $(1,\alpha)$-linked with $r$. By Lemma~\ref{vf1}, $x_1\notin V(F)$.  By the definition of $F'$, $x_1$ must be a vertex along a $\tau$-sequence $L$ such that $\varphi(rv)=\tau$ for a vertex $v$ in $N_{<\Delta}(r)$ and $\varphi(vs_h)\notin K_F$, where $K_F$ is the set of stopping colors outside of $F$. In this case, we do $\alpha\rightarrow 1$ at $x_1$.  Note that $\varphi(vs_h)$ may be changed when we did $\alpha\rightarrow 1$ at $x_1$, but we still have $\varphi(vs_h)\notin K_F$ because $1,\alpha\notin K_F$. Moreover, if $L$ contains an edge colored by $\alpha$ and a vertex $x_3\prec_L x_1$ with $\alpha\in\phibar(x_3)$ such that $x_1$ and $x_3$ are $(1,\alpha)$-linked, $x_1$ may no longer belong to the corresponding $\tau$-sequence $L$ after the color switching at $x_1$. Nonetheless, we shift the $\tau$ sequence from $v$ to $x_3$ and color $rx_3$ by $1$ if the earlier mentioned $x_3$ exists, and we shift the $\tau$ sequence from $v$ to $x_1$ and color $rx_1$ by $1$ if otherwise. Now $\tau\in\phibar(v)\cap\phibar(r)$ and $F$ is still a multi-fan.  Recall that $\beta\in\phibar(s_h)\cap K$. Let $\gamma=\varphi(vs_h)$. So $\gamma\notin K_F$. By Lemma~\ref{vf2}(a), $s_h$ and $r$ are $(\tau, \beta)$-linked. Then we do $\tau\rightarrow \beta$ at $v$.  As a result, $\beta\in\phibar(s_h)\cap\phibar(v)$ and the edge $vs_h$ is still colored by $\gamma$. If $\gamma\notin \phibar(F)$, then by $\beta\rightarrow\gamma$ at $s_h$ we have a larger multi-fan as $\gamma\notin K_F$, a contradiction with $F$ being maximum w.r.t. $e$. Thus we may assume $\gamma\in\phibar(s_i)$ for a vertex $s_i\in V(F)$.  Now we recolor the edge $vs_h$ by $\beta$, and $\gamma$ becomes a missing color at $s_h$. Since $\beta\in K$, we then have a non-elementary multi-fan containing both $s_h$ and $s_i$, a contradiction to Lemma~\ref{vf1}.

 Now we assume that $\alpha\in K_F$. Recall that $\beta\in\phibar(s_h)\cap K$. Similarly as earlier, we assume that $x_1$ is not $(\alpha,\beta)$-linked with $s_h$, and  $x_1$ is added to $F'$ through a $\tau$-sequence at $r$ starting from $rv$. We do $\alpha\rightarrow\beta$ at $x_1$. Since $x_1$ is not $(\alpha,\beta)$-linked with $s_h$ and $\alpha,\beta\in K$, this process does not change the colors on $vs_h$ and $rv$, and $x_1$ still belongs to a $\tau$-sequence at $r$ starting from $rv$. Since $s_h$ is $(1,\beta)$-linked with $r$, we can do $\beta\rightarrow 1$ at $x_1$ and this process still keeps the colors of $vs_h$ and $rv$, and $x_1$ still belongs to a $\tau$-sequence at $r$ starting from $rv$.  We then have $1\in\phibar(x_1)\cap\phibar(r)$ and returned to the previous case of $\alpha\notin K_F$ with $1$ in place of $\alpha$.

 To see (b), we just do $\gamma\rightarrow 1$ at $v_{F\cup F'}(\gamma)$ if (b) fails. Since $v_{F\cup F'}(\gamma)$ and $r$ are not $(1,\gamma)$-linked, this Kempe change does not involve any edge of $F\cup F'$. Although this Kempe change may change the color on some edge $vs_h$ where a $\varphi(rv)$-sequence is contained in $F'$ by the definition of $F'$,  since $1,\gamma\notin K_F$, we still have that the color on the edge $rv$ is not contained in $K_F$, and as a result, $v_{F\cup F'}(\gamma)$ now belongs to a non-elementary extended multi-fan (it may be different from $F\cup F'$, but it still contains the vertex $v_{F\cup F'}(\gamma)$), giving a contradiction to (a). 

The proof of (c) is similar to the proof of (b), as we can do $\gamma\rightarrow\beta$ at $v_{F'}(\gamma)$ if (c) fails. Since $\beta\in K$, $\beta\in\phibar(s_h)$ and $v_{F'}(\gamma)$ is not $(\beta,\gamma)$-linked with $s_h$, $v_{F'}(\gamma)$ now belongs to a non-elementary extended multi-fan (again it may be different from $F\cup F'$), a contradiction to (a).

If the first part of (d) fails, then by $\gamma\rightarrow 1$ at $v_{F'}(\gamma)$, we have a non-elementary extended multi-fan containing the vertex $v_{F'}(\gamma)$, a contradiction to (a). If the second part of (d) fails, then similarly we do $\gamma\rightarrow \zeta$ at $v_{F'}(\gamma)$ and get a  non-elementary extended multi-fan containing the vertex $v_{F'}(\gamma)$. Note that in both parts after the operation on $v_{F'}(\gamma)$, the extended multi-fan may be different from $F\cup F'$, but it will still contain the vertex $v_{F'}(\gamma)$.

\end{proof}
\section{Proof of Lemma~\ref{1}}
\begin{proof}
Let $G$, $k$, $s=s_0$, $r$, and the maximum multi-fan $F=(r,e_0,s_0,e_1, s_1, \ldots, e_p, s_p)$ be as defined in Lemma~\ref{1} under the coloring $\varphi\in \CC^{\D}(G-e_0)$. Recall that $K$ is the set of stopping colors at $r$ and $K_F=K-\phibar(V(F))$ is the set of stopping colors outside of $F$. Then $|K|=k$ as $r$ has core degree $k$. We denote $|K_F|$ by $k'$. To prove Lemma~\ref{1}, we assume otherwise that there are two vertices $x,x'$ with degree at least $\Delta-k+1$ such that $\alpha\in\phibar(x)\cap\phibar(x')$. Since $r$ is $(1,\alpha)$-linked with exactly one vertex, we may assume that $x$ is not $(1,\alpha)$-linked with $r$. By Lemma~\ref{vf1}, $x\notin V(F)$. Thus we do $\alpha\rightarrow 1$ at $x$. As a result, $1\in\phibar(x)$. Assume $\beta\in K\cap\phibar(s_h)$ for some $h$ with $0 \leq h\leq p$.  Let $F\cup F'$ be an extended multi-fan defined with $s_h$ where $F'$ is a collection of all the $\tau$-sequences outside of $F$ such that $\tau=\varphi(rv)$ and $v\in N_{<\Delta}(s_h)\cap N(r)$ with $\varphi(vs_h)\notin K_F$. Since $|N_{\Delta}(r)|=k$ and $rs$ is critical, all neighbors of $r$ have degree at least $\Delta-k+1$ by Theorem~\ref{thm:val}(VAL). Thus all vertices in $V(F\cup F')$ have degree at least $\Delta-k+1$.  We have the following claim.

\flushleft {\bf Claim 1.} $|N_{<\Delta}(x)\cap V(F\cup F')|\geq k+1$.

By Lemma~\ref{vf2}(d), $F$ contains at least $|\phibar(s_h)|$ many $\Delta$-neighbors of $r$. Thus we have $k'\leq k-|\phibar(s_h)|$.
Since $|\phibar(s_h)|=\Delta-d_G(s_h)$ when $h>0$ and $|\phibar(s_h)|=\Delta-d_G(s_h)+1$ when $h=0$, we have $d_G(s_h)$$\geq \Delta-|\phibar(s_h)|$, and therefore $d_G(s_h)$$\geq \Delta-(k-k')$$=\Delta-k+k'$. So $|N(r)\cap N(s_h)|$$\geq d_G(r)+d_G(s_h)-|V(G)|$$=\Delta+\Delta-k+k'-n$$\geq n/3+2k+k'$. Since there are at most $k'$ neighbors of $s_h$ that are joined to $s_h$ by colors in $K_F$, we have $|V(F\cup F')|\geq n/3+2k$. Because $r$ has exactly $k$ many $\Delta$-neighbors, $F\cup F'$ has at least $n/3+2k-k=n/3+k$ many vertices with degree less than $\Delta$. As $d_G(x)$ $\geq\Delta-k+1$ $\geq 2n/3-k+1$,  we have $|N_{<\Delta}(x)\cap V(F\cup F')|$ $\geq 2n/3-k+1+n/3+2k-n$ $=k+1$, as desired.

By considering the colors on edges joining $x$ and vertices in $N_{<\Delta}(x)\cap V(F\cup F')$, we have the following three cases.

\flushleft {\bf Case 1.} There is a vertex $u\in N_{<\Delta}(x)\cap V(F\cup F')$ such that $\tau=\varphi(xu)\in\phibar(V(F\cup F'))$.

We first assume that $u\in V(F)$. Now if $\tau\in\phibar(V(F))$ and there is a color $\gamma\in\phibar(u)$ such that $\gamma$ and $\tau$ are incomparable along $\preceq_F$,
 we just do $1\rightarrow\gamma$ at $x$ and get a contradiction with Lemma~\ref{vf2}(b), as $u$ is $(1,\gamma)$-linked with $r$ by Lemma~\ref{vf2}(a) before the operation and $u$ is $(\gamma,\tau)$-linked with $x$ in the resulting coloring. If $\tau\in\phibar(V(F))$ and there is a color $\gamma\in\phibar(u)$ such that $\tau\prec_F\gamma$ along $\preceq_F$,
 we similarly do $1\rightarrow\gamma$ at $x$ and get a contradiction with Lemma~\ref{vf2}(c). 
  In the case $\tau\in\phibar(V(F))$ and there is a color $\gamma\in\phibar(u)$ such that $\gamma\prec_F\tau$ along $\preceq_F$, we simply do a shifting from $rs$ to $rv_F(\tau)$ and uncolor the edge $rv_F(\tau)$. As a result, there exists a color in $\phibar(u)$ that is incomparable with $\tau$ and we reach an earlier case. 
  
  We now consider the case that $\tau\in\phibar(V(F'))$. If $\tau\notin K_F$, then we do $1\rightarrow \tau$ at $x$. Since $r$ is $(1,\tau)$-linked with $v_{F'}(\tau)$ by Lemma~\ref{extend}(b), and the set $\{s\in N(s_h):\varphi(ss_h)\notin K_F\}$ stays the same, it is easy to see that $F\cup F'$ is still an extended multi-fan. Similarly by Lemma~\ref{extend}(c), $s_h$ is $(\beta,\tau)$-linked with $v_{F'}(\tau)$.  We then do $\tau\rightarrow\beta$ at $x$. Note that $F$ is still a multi-fan under this new coloring. So by $\beta\rightarrow 1$ at $x$, we reach the earlier case of $\varphi(xu)\in\phibar(V(F))$, because $s_h$ is $(1,\beta)$-linked with $r$. For readers' convenience, in the remainder of this paper we will only give operations performed without repeating each time in details Lemmas~\ref{vf2} and~\ref{extend}, the set $\{s\in N(s_h):\varphi(ss_h)\notin K_F\}$, and the resulting extended multi-fan.  Now if $\tau\in K_F$, we do $1\rightarrow\beta\rightarrow\tau$ at $x$. Since $|\phibar(s_0)|\geq 2$, by Lemma~\ref{vf2}(d), there exists a color $\zeta\in\phibar(V(F))\cap K$ with $\zeta\neq\beta$.
  If $v_{F'}(\tau)$ is obtained through a linear sequence with first vertex $v$ joined to $s_h$ by the color $1$, we do $\tau\rightarrow \zeta\rightarrow 1$ at $x$ following Lemma~\ref{extend}(d) and Lemma~\ref{vf2}(a), where we reach the previous case of $\varphi(xu)\in\phibar(V(F))$. Note that here $\varphi(vs_h)$ might be changed to $\zeta$ and the set $\{s\in N(s_h):\varphi(ss_h)\notin K_F\}$ might change, but $u$ stays in an extended multi-fan in the new coloring.  If $v_{F'}(\tau)$ is obtained through a linear sequence with first vertex $v$ joined to $s_h$ by a color other than $1$, we do $\tau\rightarrow 1$ at $x$ following Lemma~\ref{extend}(d), where we reach the previous case of $\varphi(xu)\in\phibar(V(F))$. Here the set $\{s\in N(s_h):\varphi(ss_h)\notin K_F\}$ might change, but $u$ stays in an extended multi-fan in the new coloring.

We then assume that $u\in V(F')$. Let $\gamma$ be a color in $\phibar(u)$. If $\gamma\in K_F$, we can just do $1\rightarrow\beta\rightarrow\gamma\rightarrow\tau$ at $x$ and get a non-elementary extended multi-fan, a contradiction to Lemma~\ref{extend}(a). Therefore, we may assume $\gamma\notin K_F$. If $v_{F'}(\tau)$ and $u$ do not belong to the same linear sequence $L$ added to $F'$ with $u\prec_L v_{F'}(\tau)$, then we have a non-elementary extended multi-fan by $1\rightarrow \gamma\rightarrow \tau$ at $x$. Thus we may assume $v_{F'}(\tau)$ and $u$ are both belong to a linear sequence $L$ added to $F'$ with $u\prec_L v_{F'}(\tau)$. Let the first vertex of $L$ be $v$. Since $|\phibar(s_0)|\geq 2$, by Lemma~\ref{vf2}(d), there exists a color $\zeta\in\phibar(V(F))\cap K$ with $\zeta\neq\beta$. We first do $1\rightarrow\beta\rightarrow\tau$ at $x$. Now similarly as before, if $\varphi(vs_h)=1$, we then do $\tau\rightarrow\zeta\rightarrow 1$ at $x$ following Lemma~\ref{extend}(d) to reach the previous case as $\varphi(xu)=\beta\in\phibar(V(F))$.  If $\varphi(vs_h)\neq 1$, we then do $\tau\rightarrow 1$ at $x$ following Lemma~\ref{extend}(d) to reach the previous case as $\varphi(xu)=\beta\in\phibar(V(F))$.

\flushleft{\bf Case 2.} There is a vertex $u\in N_{<\Delta}(x)\cap V(F\cup F')$ such that 
 $\tau=\varphi(xu)\notin\phibar(V(F\cup F'))$ and not all the $\tau$-sequence outside of $F$ is of Type D. 

By the assumption of this case,  there is a $\tau$-sequence $L$ outside of $F$ which is of Type A, B, or C. Let $\gamma\in\phibar(u)$. We first assume that $u\in V(F)$. Now if $L$ is of Type A or C, then by doing $1\rightarrow\gamma\rightarrow\tau$ at $x$, we have a non-elementary multi-fan contradicting Lemma~\ref{1}. If $L$ is of Type B with $\phibar(V(L))\cap \{1,\gamma\}\neq\emptyset$, then by doing $1\rightarrow\gamma\rightarrow\tau$ at $x$, we still have $\phibar(V(L))\cap \{1,\gamma\}\neq\emptyset$, reaching a contradiction by resulting either a larger multi-fan or a non-elementary multi-fan. Now the remaining case is that $L$ is of Type B, and there exists a color $\eta\in\phibar(V(F))\cap\phibar(V(L))$ with $\eta\notin\{1,\gamma\}$. If $u\npreceq_F v_F(\eta)$ along $\preceq_{F}$, then by doing $1\rightarrow\gamma\rightarrow\tau$ at $x$, we have a non-elementary multi-fan contradicting Lemma~\ref{1}. If $u\preceq_F v_F(\eta)$ along $\preceq_{F}$, then we  simply do a shifting from $rs$ to $rv_F(\eta)$ and uncolor the edge $rv_F(\eta)$, reaching the earlier case of $u\npreceq_F v_F(\eta)$.

We then assume $u\in V(F')$. We first consider the case that $L$ is not of Type B with $\{1\}=\phibar(V(L))\cap\phibar(V(F))$, or $L$ is of Type B with $\{1\}=\phibar(V(L))\cap\phibar(V(F))$ and $x\in V(L)$. If $\gamma\notin K_F$, we just do $1\rightarrow\gamma\rightarrow\tau$ at $x$ to get a non-elementary extended multi-fan, a contradiction to Lemma~\ref{extend}(a). Therefore, we have $\gamma\in K_F$. If $L$ is not of Type B with $\{\beta\}=\phibar(V(L))\cap\phibar(V(F))$, we have a non-elementary extended multi-fan by $1\rightarrow\beta\rightarrow\gamma\rightarrow \tau$ at $x$, a contradiction to Lemma~\ref{vf2}(a). Thus we may assume $L$ is of Type B with $\{\beta\}=\phibar(V(L))\cap\phibar(V(F))$. Suppose that $u$ is added to $F'$ by a linear sequence $L'$ with first vertex $v$. Again since $|\phibar(s_0)|\geq 2$, by Lemma~\ref{vf2}(d), there exists a color $\zeta\in\phibar(V(F))\cap K$ with $\zeta\neq\beta$. Now if $\varphi(vs_h)\neq 1$, we do $1\rightarrow\gamma\rightarrow\tau$ at $x$ to get a non-elementary extended multi-fan, and if $\varphi(vs_h)=1$, we do $1\rightarrow\zeta\rightarrow\gamma\rightarrow\tau$ at $x$ to get a non-elementary extended multi-fan, both give contradictions to Lemma~\ref{vf2}(a).

Thus we can assume $L$ is of Type B with $\{1\}=\phibar(V(L))\cap\phibar(V(F))$ and there is a vertex $z\in V(L)$ with $1\in\phibar(z)$ and $z\neq x$. Clearly $\tau\notin K_F$ as $L$ is of Type B. By the definition of maximal linear sequences outside of $F$, $z$ is the last vertex of $L$.  Let the first vertex of $L$ be $w$. Note that here $z$ and $w$ could be the same vertex.
Clearly $F'$ and $L$ do not share common vertices, as otherwise $z\in V(F')$ and we have a non-elementary extended multi-fan. Recall that $\gamma\in\phibar(u)$. Now we do $1\rightarrow\beta$ at both $x$ and $z$ following Lemma~\ref{vf2}(a). As a result, $\beta\in\phibar(z)\cap\phibar(x)$. Note that $\varphi(rw)=\tau$ and $d(w)<\Delta$, so $\tau\notin K_F$. Therefore, $s_h$ and $r$ must be $(\beta,\tau)$-linked, as otherwise we have a larger multi-fan by interchanging $\beta$ and $\tau$ along $C_{r}(\beta,\tau)$. Now if $s_h$ and $x$ are $(\beta,\tau)$-linked, we can do $\beta\rightarrow\tau$ at $z$ and then do $\beta\rightarrow 1$ at $x$, reaching the earlier case that $L$ is not of Type B. We then assume $s_h$ and $z$ are $(\beta,\tau)$-linked. In this case, we first do $\beta\rightarrow \tau$ at $x$ and then do $\beta\rightarrow 1$ at $z$. Now $\varphi(ux)=\beta$, $\tau\in\phibar(x)$, and $1\in\phibar(z)$. We then do a shift along $L$ from $w$ to $z$, and color the edge $rz$ by $1$. For reader's convenience, we switch labels for color $1$ and $\tau$ to meet the notations we used earlier. After the switching of $1$ and $\tau$, we now still have $1\in\phibar(r)\cap\phibar(x)$, $\gamma\in\phibar(u)$, $\varphi(ux)=\beta$, and $F\cup F'$ is still an extended multi-fan as $\tau\notin K_F$, which returns to Case 1. Finally we may assume that $s_h$ is $(\beta,\tau)$-linked with neither $z$ nor $x$. We then do $\beta\rightarrow\tau$ at both $z$ and $x$. As a result, $\tau\in\phibar(x)\cap\phibar(z)$ and $\varphi(ux)=\beta$. Recall that $\tau\notin K_F$. If $r$ is $(1,\tau)$-linked with $z$, then by doing $\tau\rightarrow 1$ at $x$, we reach Case 1. If $r$ is $(1,\tau)$-linked with $x$, then by doing $\tau\rightarrow 1$ at $z$, we reach the earlier case where we shift along $L$ from $w$ to $z$, and color the edge $rz$ by $1$. If $r$ is $(1,\tau)$-linked with neither $z$ nor $x$, then by doing $\tau\rightarrow 1$ at both $z$ and $x$, we reach Case 1.
 This finishes Case 2.

\flushleft {\bf Case 3.} All the $\tau$-sequence outside of $F$ is of Type D, where $\tau=\varphi(xu)\notin$ $\phibar(V(F)$ $\cup $ $ V(F'))$ and $u\in N_{<\Delta}(x)\cap V(F\cup F')$.

Since $|N_{<\Delta}(x)\cap V(F\cup F')|\geq k+1$ by Claim 1 and $|K_F|<k$, there must exist two vertices $u$ and $u^*$ in $N_{<\Delta}(x)\cap V(F\cup F')$ such that $\tau=\varphi(xu)\notin\phibar(V(F)\cup V(F'))$ and $\tau^*=\varphi(xu^*)\notin\phibar(V(F)\cup V(F'))$ where the $\tau$-sequence $L_1$ and $\tau^*$-sequence $L_2$ are both of Type D ending with the same stopping color of $F$ in $K_F$.
 
 We claim that one of $L_1$ and $L_2$ is a sub-sequence of the other one. Otherwise, since $L_1$ and $L_2$ are of Type D both ending with the same stopping color, there exists $\theta\in\phibar(v_1)\cap\phibar(v_2)$ such that $v_1\in V(L_1)$, $v_2\in V(L_2)$, and $v_1\neq v_2$. Because $L_1$ and $L_2$ are of Type D, $\theta\notin\phibar(V(F))$. Since $\beta\in\phibar(s_h)$, at most one of $v_1$ and $v_2$ is $(\beta,\theta)$-linked with $s_h$. Thus we may assume that $v_1$ is not $(\beta,\theta)$-linked with $s_h$. Note that if $\theta\notin K_F$, $r$ and $s_h$ must be $(\beta,\theta)$-linked, as otherwise by switching $\beta$ and $\theta$ along $C_{r}(\beta,\theta,\varphi)$, we would have a larger multi-fan. Now by $\theta\rightarrow\beta$ at $s_1$, we have reached Case 2 as the sub-sequence of $L_1$ ending at $v_1$ is of Type B and this operation will not change the set $\{s\in N(s_h):\varphi(ss_h)\notin K_F\}$. Thus we have as claimed.
 
 Now by the above claim, we may assume $L_2$ is a sub-sequence of  $L_1$, $\tau^*\in\phibar(z)$ with $z\in V(L_1)$, and $\tau\notin K_F$. We are going to consider the following three cases depending on which one of $u$ and $u^*$ is in $V(F')$.

 We first assume that both $u$ and $u^*$ are in $V(F)$. Similarly as earlier, we may assume that there exist $\gamma\in\phibar(u)$ and $\gamma^*\in\phibar(u^*)$ such that $\gamma$ and $\gamma^*$ are incomparable along $\preceq_F$. Otherwise, say $\gamma\preceq_F \gamma^*$, then by shifting from $s_0$ to $v_F(\gamma^*)$ and uncolor the edge $rv_F(\gamma^*)$, we have as desired. Now we first do $1\rightarrow\gamma\rightarrow\tau$ at $x$ following Lemma~\ref{vf2}(a) and consider a maximal multi-fan $F^*$ under this new coloring $\varphi$. Since now $\tau\in\phibar(u)$, and $\gamma$ and $\gamma^*$ were incomparable along $\preceq_F$ earlier, we have $u,u^*,z\in V(F^*)$, $\tau\in\phibar(u)\cap\phibar(x)$,  $\tau^*\in\phibar(z)$, $\varphi(xu^*)=\tau^*$, and $\tau^*$ and $\gamma^*$ are incomparable along $\preceq_{F^*}$. Following Lemma~\ref{vf2}(a), we can do $\tau\rightarrow 1\rightarrow\gamma^*$ at $x$. As a result, $F^*$ is still a multi-fan and now $u^*$ and $x$ are $(\gamma^*,\tau^*)$-linked, a contradiction to Lemma~\ref{vf2}(b).

 We then assume that $u\in V(F')$ and $u^*\in V(F)$, or both $u, u^*\in V(F')$ but $u$ and $u^*$ do not satisfy $u\preceq_{L^*} u^*$ for a linear sequence $L^*$ in $F'$. Note that $\{\gamma,\gamma^*\}\cap \{\tau,\tau^*\}=\emptyset$ as $\tau=\varphi(xu)\notin\phibar(V(F)\cup V(F'))$ and  $\tau^*=\varphi(xu^*)\notin\phibar(V(F)\cup V(F'))$. In the case that $\gamma\notin K_F$, we do $1\rightarrow \gamma\rightarrow\tau$ at $x$ following Lemma~\ref{extend}(b), and in the case that $\gamma\in K_F$, we do $1\rightarrow \beta\rightarrow\gamma\rightarrow\tau$ at $x$ following Lemma~\ref{vf2}(b) and Lemma~\ref{extend}(c). Note that the set $\{s\in N(s_h):\varphi(ss_h)\notin K_F\}$ does not change after the above operations and $\tau\in\phibar(u)$ now.  Thus if $\gamma$ was in $\phibar(V(L_1))$ and $\gamma$ was changed to $1$ or $\beta$ by the above operations in  $\phibar(V(L_1))$, we then have a non-elementary extended multi-fan contradicting Lemma~\ref{extend}(a). Otherwise, since $L_2$ is a sub-sequence of  $L_1$, $\tau^*\in\phibar(z)$ with $z\in V(L_1)$ and  $u\npreceq_{L^*} u^*$ for any $L^*$ in $F'$, the new extended multi-fan will contain $z$, $u$, $u^*$  while $\tau^*\in\phibar(z)$, $\tau\in\phibar(x)$ and $\varphi(u^*x)=\tau^*$. Now since $\tau\notin K_F$, we just do $\tau\rightarrow 1$ at $x$ following Lemma~\ref{extend}(b) to reach Case 1 as $\varphi(u^*x)=\tau^*\in\phibar(z)$.
 
 Finally we assume that $u^*\in V(F')$ and $u\in V(F)$, or both $u, u^*\in V(F')$ but $u$ and $u^*$ do not satisfy $u^* \preceq_{L^*} u$ for a linear sequence $L^*$ in $F'$. Note that $\{\gamma,\gamma^*\}\cap \{\tau,\tau^*\}=\emptyset$ as $\tau=\varphi(xu)\notin\phibar(V(F)\cup V(F'))$ and  $\tau^*=\varphi(xu^*)\notin\phibar(V(F)\cup V(F'))$.  Similar as before, in the case that $\gamma^*\notin K_F$, we do $1\rightarrow \gamma^*\rightarrow\tau^*$ at $x$ following Lemma~\ref{extend}(b), and in the case that $\gamma^*\in K_F$, we do $1\rightarrow \beta\rightarrow\gamma^*\rightarrow\tau^*$ at $x$ following Lemma~\ref{vf2}(b) and Lemma~\ref{extend}(c). Note that the set $\{s\in N(s_h):\varphi(ss_h)\notin K_F\}$ does not change after the above operations, and $\tau^*\in\phibar(u^*)$ and $\tau^*\in\phibar(x)$. Moreover, we now have $\tau^*\in\phibar(u^*)$, $\tau^*\in\phibar(z)$, $\varphi(u^*x)=\gamma^*$, and vertices $u$, $u^*$ and the sub-sequence of $L_1$ after $z$ is contained in an extended multi-fan after the above operations. 
   Thus if $\gamma^*$ was a missing color in the sub-sequence of $L_1$ after $z$ and $\gamma^*$ was changed to $1$ or $\beta$ by the above operations as a missing color in the sub-sequence of $L_1$, we then have a non-elementary extended multi-fan contradicting Lemma~\ref{extend}(a). 
    If $\gamma^*$ was a missing color in the sub-sequence of $L_1$ until $z$  and $\gamma^*$ was changed to $1$ or $\beta$ by the above operations  as a missing color in the sub-sequence of $L_1$, we then do $\tau^*\rightarrow 1$ at $x$ when $\tau^*\notin K_F$ and $\tau^*\rightarrow\beta\rightarrow 1$ at $x$ when $\tau^*\in K_F$. Now $1\in\phibar(x)$, $\tau^*\in\phibar(u)$, $\varphi(xu)=\tau$, $u$ and $u^*$ are in an extended multi-fan while there is $\tau$-sequence of either Type A or B, reaching Case 1. If the above two possibilities did not happen, then we still have $\tau^*\in\phibar(u^*)\cap\phibar(z)\cap\phibar(x)$, $\tau=\phibar(ux)$, $u$ and $u^*$ are still in an extended multi-fan, and the $\tau$-sequence $L_1$ stays the same containing the vertex $z$. Now if $\tau^*\notin K_F$, we do $\tau^*\rightarrow 1$ at both $x$ and $z$ following Lemma~\ref{extend}(b) to reach Case 1, as now there is a $\tau$-sequence of Type B.   In the case that $\tau^*\in K_F$,  we do $\tau^*\rightarrow\beta\rightarrow 1$ at both $x$ and $z$ following Lemma~\ref{extend}(c) and Lemma~\ref{vf2}(b) to reach Case 1, as now there is a $\tau$-sequence of Type B.

    \end{proof}

    \bibliographystyle{plain}
\bibliography{SSL-BIB_08-19}
\end{document}